\documentclass[a4paper]{article}
\usepackage[margin=1in]{geometry}

\title{Disjoint Total Dominating Sets in Near-Triangulations}

\usepackage{authblk}

\author{
    P. Francis\footnote{
        Part of the work was done as a post-doctoral fellow at Indian Institute of Technology Palakkad.
        }
    }
\affil{
    Department of Mathematics\\
    SAS, VIT-AP University, Amaravati, 
    Andhra Pradesh, India.\\
    francis@vitap.ac.in
    }
\author{
    Abraham M. Illickan,
    Lijo M. Jose, and 
    Deepak Rajendraprasad
    }
\affil{
    Department of Computer Science\\
    Indian Institute of Technology Palakkad\\
    111901003@smail.iitpkd.ac.in, 
    112004005@smail.iitpkd.ac.in,
    deepak@iitpkd.ac.in
    }

\usepackage{amsthm} 
    \theoremstyle{plain}
        \newtheorem{theorem}{Theorem}
        \newtheorem{corollary}[theorem]{Corollary}
        \newtheorem{lemma}[theorem]{Lemma}
        \newtheorem{observation}[theorem]{Observation}
        \newtheorem{conjecture}{Conjecture}
    \theoremstyle{definition}
        \newtheorem{definition}[theorem]{Definition}
        \newtheorem{construction}[theorem]{Construction}
    \theoremstyle{remark}
        \newtheorem*{remark}{Remark}
   
    \usepackage{amssymb}
	    
	    \newcommand{\set}[1]{\left\{ #1 \right\}}
	    \newcommand{\card}[1]{\left\vert #1 \right\vert}
	    \newcommand{\indsg}[1]{\left\langle #1 \right\rangle} 

\usepackage{hyperref}
	\hypersetup{
	    colorlinks=true,
	    linkcolor=red,
	    citecolor=red,
	    pdfborder={0 0 0}}
	    
\begin{document}
\maketitle
\renewcommand{\baselinestretch}{1.2}\normalsize

\begin{abstract}
    We show that every simple planar near-triangulation with minimum degree at least three contains two disjoint total dominating sets. The class includes all simple planar triangulations other than the triangle. This affirms a conjecture of Goddard and Henning [Thoroughly dispersed colorings, J. Graph Theory, 88 (2018) 174--191].

\end{abstract}

\section{Introduction}

All graphs considered in this paper are finite, simple and undirected. A \emph{near-triangulation} is a simple planar graph embedded in the plane such that all its faces except possibly the outer one are bounded by three edges. The most general form of the result we establish in the paper is

\begin{theorem}                                     \label{thm:General}
    Let $G$ be a near-triangulation and $V'$ be a subset of vertices of $G$ containing all the vertices of degree at least three and at most two vertices of degree two. Then, there exists two disjoint subsets $V_1$ and $V_2$ of $V(G)$ such that each vertex $v \in V'$ has at least one neighbor each in $V_1$ and $V_2$.
\end{theorem}

If we place a restriction on the minimum degree, the above result can be stated in the language of total domination. Let $G(V,E)$ be a graph. For $S \subset V$, the \emph{open neighborhood} $N(S)$ of $S$ is the set of all vertices in $G$ which have at least one neighbor in $S$. The set $S$ is called \emph{dominating} if $N(S) \cup S = V$ and \emph{total dominating} if $N(S) = V$. The minimum size of a dominating (resp. total dominating) set is called the \emph{domination number} $\gamma(G)$ (resp. \emph{total domination number} $\gamma_t(G)$) of $G$, while the maximum number of pairwise disjoint dominating (total dominating) sets in $G$ is called the \emph{domatic number} $d(G)$ (\emph{total domatic number} $d_t(G)$) of $G$. A \emph{triangulated disk} (resp. \emph{triangulation}) is a near-triangulation in which the unbounded face is bounded by a simple cycle (resp. triangle). The following two corollaries follow by restricting Theorem~\ref{thm:General} to the respective graph classes.

\begin{corollary}[The Goddard-Henning Conjecture, Conjecture $30$, \cite{god}]      \label{corr:GodHenTriangulation}
    If $G$ is a planar triangulation of order at least four, then $d_t(G) \geq 2$.
\end{corollary}
\begin{corollary}[Speculated in \cite{god}]                                   \label{corr:GodHenDisc}
    If $G$ is a near-triangulation with minimum degree at least three, then $d_t(G) \geq 2$.
\end{corollary}

Both the above results were speculated by Goddard and Henning~\cite{god} in 2018. They had termed the first one as one of the ``most frustrating'' questions to them and it indeed served as our primary target in this investigation. There were several attempts made on the Goddard-Henning Conjecture which affirmed it on many interesting classes of triangulations. This includes Hamiltonian triangulations (Nagy \cite{nag}), triangulations with all odd-degree vertices, triangulations with a Hamiltonian dual (Goddard and Henning \cite {god}), triangulations with at most two vertices of degree at most four, and triangulations with a $2$-factor in which no cycle has length congruent to $2$ modulo $4$. (B\'{e}rczi and G\'{a}bor \cite{bergab}). Some of these attempts also reformulated the conjecture in various equivalent and slightly stronger ways. The interested reader is invited to check \cite{bergab} for a nice catalog. 

Theorem~\ref{thm:General} is tight in two senses. Firstly we cannot increase the number of degree two vertices in $V'$. The $3$-sun, which is a graph obtained by adding a triangle of chords to a six-cycle, does not have two disjoint dominating sets \cite{god}. Secondly, the result cannot be extended to general graphs as shown by Zelinka \cite{zel2}. He observed that for any positive integer $k$, the incidence graph of the complete $k$-uniform hypergraph $H$ on $n$ vertices with $n \geq 2k-1$ does not have two disjoint total dominating sets even though the minimum degree is high ($k$). But we do not know whether Theorem~\ref{thm:General} can be extended to planar graphs which are not near-triangulations. Goddard and Henning \cite{god}  had also shown that $d_t(G) \leq 4$ for every planar graph $G$ and hence conditions which ensure $d_t(G) \geq 3$ are equally interesting, but we haven't been able to make any progress there yet. Goddard and Henning \cite{god} conjecture that every triangulation $G$ with minimum degree four has $d_t(G) \geq 3$. 

There are at least two other ways in which total domatic number is studied in literature. A $k$-coloring of the vertices of a graph $G$ is called a  \emph{$k$-coupon coloring} if every vertex sees at least one vertex of each color in its open neighborhood \cite{chen}. This is also known as a \emph{total domatic coloring} \cite{francis2021coupon, francis2021domatic} and a \emph{thoroughly dispersed coloring} \cite{god}. The \emph{coupon chromatic number} $\chi_c(G)$ of a graph $G$ is the maximum $k$ for which $G$ has a $k$-coupon coloring. It is easy to see that $d_t(G) = \chi_c(G)$ since every color class in a coupon coloring has to be a total dominating set. In fact, in this paper we present our arguments in the form of vertex coloring (Theorem~\ref{thm:NearCouponColouring}). A hypergraph $H$ has a proper $k$-coloring if there is a $k$-coloring of the vertices of $H$ such that every hyperedge contains all the $k$ colors\footnote{This is not the standard definition. The more common definition is to keep a weaker demand that no hyperedge is monochromatic. The two definitions agree when $k = 2$ and $2$-colorability of hypergraph is also known as \emph{Property B}.}. Given a graph $G(V,E)$, its \emph{open neighborhood hypergraph} is the hypergraph $H$ on the same vertex set $V$ in which the open-neighborhood of each vertex in $G$ is a hyperedge in $H$. Hence a $k$-coupon coloring of $G$ corresponds to a proper $k$-coloring of the open-neighborhood hypergraph of $G$. In particular, Corollary~\ref{corr:GodHenDisc} can be seen as a guarantee that the open-neighborhood hypergraph of a triangulated disk with minimum degree at least three has Property~B.

There is considerable literature on total domination in graphs. See for instance, \cite{bre,grav2,hen2,ho} and a survey of selected topics by Henning \cite{hen1}. The concept of domatic number and total domatic number were introduced by Cockayne et al., in \cite{cok} and \cite{cok1} respectively, and are investigated further as follows. In \cite{zel3}, Zelinka obtained the characterization of $r$-regular bipartite (both directed and undirected) graphs with $d_t(G)=r$. Akbari et al., \cite{akb} provided a criterion for cubic graphs which have total domatic number at least two but the same problem is NP-complete even for bipartite graphs (Heggernes and Telle \cite{heg}). 
Also, Henning and Peterin \cite{hen3} provided a constructive characterization of graphs that have two disjoint total dominating sets. Aram et al., \cite{ara} shown that the total domatic number of a random $r$-regular graph is almost surely at most $r-1$,
and they gave a lower bound on the total domatic number of a graph in terms of order, minimum degree and
maximum degree. Chen et al., \cite{chen} shown that every $r$-regular graph has $d_t(G)\geq (1 -  o(1))r/ \log r$ as $r\rightarrow\infty$, and the proportion of $r$-regular graphs for which $d_t(G) \leq (1 + o(1))r/ \log r$ tends to $1$ as $|V(G)|\rightarrow\infty$. Bouchemakh and Ouatiki \cite{bou} studied the domatic and the total domatic numbers of the $2$-section graph
of the order-interval hypergraph of a finite poset. 
In \cite{koi}, Koivisto et al., showed that it is NP-complete to decide whether $d_t(G)\geq3$ where $G$ is a bipartite planar graph of bounded maximum degree. Also, they have shown that if $G$ is split or $k$-regular graph for $k\geq3$, then it is NP-complete to decide whether $d_t(G) \geq k$.  The first and last author of this paper have studied the domatic and total domatic number of Cartesian product graphs \cite{francis2021coupon, francis2021domatic}. In \cite{MathTar}, Matheson and Tarjan showed that if $G$ is a triangulated disk, then $d(G)\geq 3$ and conjectured that for large enough $n$, every triangulation on $n$ vertices has a dominating set of size at most $n/4$. This conjecture is still open. 

\subsection*{Terminology and notation}

Let $G$ be a graph and $k \in \mathbb{N}$. The vertex-set and edge-set of $G$ are denoted respectively by $V(G)$ and $E(G)$. The subgraph of $G$ induced on a set $S \subset V(G)$ is denoted by $\indsg{S}$. The \emph{degree} $d_G(v)$ of a vertex $v$ in $G$ is $|N_G(v)|$. The subscripts may be omitted when the graph is clear from the context. The minimum and maximum degree of  $G$ are denoted, respectively by $\delta(G)$ and $\Delta(G)$. A vertex of degree exactly $k$, at least $k$ and at most $k$ in $G$ are respectively termed \emph{$k$-vertex}, \emph{$k^+$-vertex} and \emph{$k^-$-vertex}. For a vertex $v$  (resp. edge $e$) of $G$, the subgraph of $G$ obtained by removing $v$ (resp. $e$) from $G$ is denoted by $G \setminus v$ (resp. $G \setminus e$). A \emph{cut} in $G$ is a partition of $V(G)$ into two disjoint subsets. An edge of $G$ with one endpoint in each part of the cut is said to \emph{cross} the cut.

Let $K_n$, $P_n$ and $C_n$, respectively denote the complete graph, the simple path and the simple cycle on $n$ vertices. The graph obtained by adding a universal vertex to $P_n$ (resp. $C_n$) is called the \emph{$n$-fan $F_n$} (resp. \emph{$n$-wheel} $W_n$). A diamond is a $4$-cycle $(v_1, v_2, v_3, v_4)$ together with a chord $v_1v_3$. A $3$-sun is a $6$-cycle $(v_1, v_2, v_3, v_4, v_5, v_6)$ together with three chords forming a triangle $(v_1, v_3, v_5)$.

In a near-triangulation $G$, we refer to the cycle bounding the unbounded face as the \emph{boundary} $B(G)$ of $G$ and the vertices and edges in $B(G)$ as \emph{boundary vertices} and \emph{boundary edges} of $G$. The remaining vertices and edges are called \emph{internal}. An internal edge between two boundary vertices is called a \emph{chord}. 

Given a (partial) $k$-coloring of the vertices of $G$, a vertex $v$ is said to be \emph{satisfied} if $N(v)$ contains at least one vertex of each color. Given a two-coloring of some set $X$, we call the the process of swapping the color of each vertex in $X$ as \emph{flipping} the colors.  

\section{Proof of Theorem~\ref{thm:General}}

In this section, we prove the following theorem which is a restatement of Theorem~\ref{thm:General} in the language of vertex coloring. However this is stronger than Corollary~\ref{corr:GodHenDisc} since we handle all near-triangulations. The strengthening helps us run a proof by structural induction since near-triangulations, unlike triangulated disks, is a family that is closed under deletion of vertices and edges from the boundary. On the other hand, if one restricts to triangulations (Corollary~\ref{corr:GodHenTriangulation}), then the initial observations in this section are unnecessary. We will say more about this simplification after Observation~\ref{obs:CommonSecondneigbor}.

\begin{theorem}                                     \label{thm:NearCouponColouring}
    Let $G$ be a near-triangulation. Let $T$ be the set of all $3^+$-vertices in $G$ and let $S$ be any subset of $2$-vertices in $G$ such that $|S| \leq 2$. There exists a two-coloring of $V(G)$ such that each vertex $v \in T \cup S$ sees both the colors in $N(v)$.
\end{theorem}

Till we complete the proof of Theorem~\ref{thm:NearCouponColouring}, we call near-triangulations which satisfy the theorem as \emph{good} and others as \emph{bad}. Given a near-triangulation $G$ and a subset $S$ of $2$-vertices in $G$, a two-coloring which satisfies all the $3^+$-vertices and the vertices in $S$ is called a \emph{good coloring} of $(G, S)$. The vertices in $S$ will be called \emph{special}. 

For the rest of this section, we fix $G(V,E)$ to be a bad near-triangulation with the smallest $\card{V} + \card{E}$. We also fix $S$ to be an arbitrary subset of $2$-vertices of $G$ such that $\card{S} \leq 2$. We will show that $(G,S)$ has a good coloring, contradicting the existence of $G$. The minimality in the choice of $G$ helps us to simplify the following observations.

\begin{observation}                                 \label{obs:2Connected}
    $G$ is $2$-connected.
\end{observation}

\begin{proof}
    If $G$ is disconnected, then each component of $G$ is smaller than $G$ and hence good. In this case, $G$ is easily seen to be good. 

    Suppose $G$ contains a bridge $e=uv$ and let $G_u$ and $G_v$ be the two components of $G \setminus e$. For $x \in \{u, v\}$, let $S_x = S \cap V(G_x)$. Without loss of generality we can assume that $\card{S_v} \leq 1$, and let $S'_v = S_v \cup \set{v}$, if $d_{G_v}(v)=2$ and $S'_v = S_v$ otherwise. By the minimality of $G$, both $G_u$ and $G_v$ are good. If $d_G(u)>3$ or $d_G(u)=1$  then a good coloring of $(G_u, S_u)$ together with a good coloring of $(G_v, S'_v)$ will be a good coloring of $(G,S)$. If $d_G(u) \in \{2,3\}$ then the above procedure will still give a good coloring of $(G,S)$ provided we flip the coloring of $G_v$ if $u$ does not see both colors in the first coloring.

    Suppose $G$ contains a cut vertex $v$. We can consider $G$ as two smaller graphs $G_1$ and $G_2$ which share exactly one common vertex $v$. For $i \in \{1, 2\}$, let $S_i = S \cap V(G_i)$ and let $d_i = d_{G_i}(v)$. Since $G$ is bridgeless, $d_i \geq 2$. Without loss of generality we can assume $\card{S_2} \leq 1$ and let $S'_2 = S_2 \cup \set{v}$ if $d_2 = 2$ and $S'_2 = S_2$ otherwise. A good coloring of $(G_1, S_1)$ can be combined with a good coloring of $(G_2, S'_2)$, flipping the coloring of $G_2$ if necessary to match the color of $v$ in both colorings, to obtain a good coloring of $(G,S)$.
\end{proof}
 
\begin{remark}
    Insisting that a good coloring should satisfy a set $S$ of $2$-vertices along with all the $3^+$-vertices not just strengthened Theorem~\ref{thm:NearCouponColouring} marginally, but also helped us critically in establishing Observation~\ref{obs:2Connected}. In fact, we could not bypass this strengthening even if we restricted to $2$-connected near-triangulations (triangulated disks) since we need closure of the graph class under deletion of boundary vertices and edges for some of the further observations (c.f. Observation~\ref{obs:NoConsecutive4}) too. 
\end{remark} 
 
Since $G$ is a $2$-connected near-triangulation, it is a triangulated disk. The only triangulated disks on at most $4$ vertices are $K_3$, $K_4$ and the diamond. We can easily verify that all three of them are good. Henceforth we assume that $G$ is a triangulated disk with at least $5$ vertices.

\begin{observation}                     \label{obs:NoConsecutive4}
       There are no consecutive $4^+$-vertices on the boundary of $G$.
\end{observation}
\begin{proof}
    Let $e$ be the boundary edge between such a pair of consecutive vertices. A good coloring of $(G \setminus e, S)$, which exists by the minimality of $G$, will be a also a good coloring of $(G, S)$. 
\end{proof}

\begin{observation}                     \label{obs:NoConsecutive23-}
    No $2$-vertex in $G$ has a $3^-$-neighbor.
\end{observation}
\begin{proof}
    Let $v$ be a $2$-vertex in $G$ with neighbors $u$ and $w$. Since $G$ is a simple triangulated disk on at least $5$ vertices, all the three vertices $u$, $v$ and $w$ are on the boundary of $G$ and there exists an edge $uw$ which is an internal in $G$.  Hence $\set{u, w}$ has a second common neighbor $x$ and $\indsg{\set{u,v,w,x}}$ is a diamond with $uw$ as the chord. In particular, both $u$ and $w$ are $3^+$-vertices. Let one of them, say $w$, be a $3$-vertex. Then the edge $wx$ is also a boundary edge of $G$ and hence $H = G \setminus \set{v, w}$ is also a near-triangulation. We can extend a good coloring of $(H, S \setminus \set{v})$ to a good coloring of $(G, S)$ by giving $v$ the color different from that of $x$ and $w$ the color different from that of $u$.
\end{proof}

\begin{observation}                                 \label{obs:No2vertexOutsideS}
        $G$ has no $2$-vertices outside $S$.
\end{observation}
\begin{proof}
    Suppose $G$ has a $2$-degree vertex $v \notin S$ and let $u$ and $w$ be its neighbors. By Observation~\ref{obs:NoConsecutive23-}, both $u$ and $w$ are $4^+$-vertices. Hence any good coloring of $(G \setminus \set{v}, S)$ can be extended to a good coloring of $(G, S)$ by giving any color to $v$. Notice that we do not need to satisfy $v$ since it is a $2$-vertex outside $S$.
\end{proof}

Due to Observation~\ref{obs:No2vertexOutsideS}, we do not need to specify $S$ separately anymore. $S$ will be the set of all $2$-vertices in the triangulated disk $G$. Moreover any good coloring of $(G,S)$ is an ordinary $2$-coupon coloring of $G$ since it satisfies all the vertices of $G$.

\begin{observation}                                 \label{obs:No3-3Chord}
       There is no chord between two $3$-vertices on the boundary of $G$.
\end{observation}
\begin{proof}
    Let $uv$ be a chord between two $3$-degree vertices $u$ and $v$ on the boundary of $G$. As $uv$ is an internal edge, it will be a part of two triangular faces $\indsg{\set{u,v,w_1}}$ and $\indsg{\set{u,v,w_2}}$. Since $u$ and $v$ are $3$-vertices, $w_1$ and $w_2$ are their only neighbors other than each other. In this case all the four edges of the cycle $C = (u, w_1, v, w_2, u)$ are on the boundary of $G$. Since $G$ is a triangulated disk,  (Observation~\ref{obs:2Connected}), $C$ is the entire boundary of $G$, and hence $G$ is a diamond. This contradicts our assumption that $G$ had at least $5$ vertices.
\end{proof}

Next we construct a special independent set $I$ in $G$ as follows.

\begin{construction}[$I$]                           \label{const:I}
    Start with $I = \emptyset$. In Round~1, for each $4^+$-vertex on the boundary $B(G)$ of $G$, we add its clockwise next boundary vertex to $I$. In Round~2, we enlarge $I$ to a maximal independent set of $3^-$-vertices from $B(G)$. In Round~3, we enlarge $I$ to a maximal independent set of $4^-$-vertices from $G$.
\end{construction}

\begin{observation}                                 \label{obs:Ivertices}
    $I$ is a maximal independent set of $4^-$-vertices in $G$ which contains all the $2$-vertices and none of the $4^+$-vertices from the boundary of $G$.
\end{observation}
\begin{proof}
    By Observations \ref{obs:NoConsecutive4} and \ref{obs:No3-3Chord}, $I$ is an independent set after Round~1. It remains so, by construction, after the subsequent two rounds. By Observation~\ref{obs:NoConsecutive23-}, all the $2$-vertices in $B(G)$ are added to $I$ in Round~1. By Observation~\ref{obs:NoConsecutive4}, no $4^+$-vertices in $B(G)$ are added to $I$ in Round~1 and since they are all dominated by $I$ after Round~1, by construction, they are not added to $I$ in any of the two subsequent rounds.
\end{proof}

\begin{observation}                     \label{obs:CommonSecondneigbor}
    If $v \in  I$, then every pair of vertices in $N(v)$ has a second common neighbor in $G$.
\end{observation}
\begin{proof}
    Let $v$ be a $2$-vertex in $I$ with neighbors $u$ and $w$. Since $G$ is a triangular disk with at least $5$ vertices, $uw$ is a chord of $G$ and hence $\set{u,w}$ have a second common neighbor. 
    
    Let $v$ be a boundary $3$-vertex in $I$ with neighbors $u$, $v'$ and $w$, where $u$ and $w$ are in $B(G)$. Since $v$ has degree exactly three, $v'$ is a common neighbor of $u$ and $w$. If $uv'$ is a boundary edge of $G$, $u$ would be a $2$-vertex. By Observation~\ref{obs:Ivertices}, $u \in I$ in which case $v$ would not have been in $I$. Hence $uv'$ is an internal edge of $G$ and hence $\set{u, v'}$ has a common neighbor other than $v$. The case for $\set{v', w}$ is similar.
    
    Finally let $v$ be any $3$-degree or $4$-degree internal vertex in  $I$.  Let $u$ and $w$ be two distinct vertices in $N(v)$. If $u$ and $w$ are non-adjacent in $G$, then $v$ is a $4$-vertex, and hence both the remaining vertices in $N(v)$ are common neighbors of $\set{u, w}$. If $uw$ is an internal edge of $G$, then since $G$ is a near triangulation, $\set{u,w}$ has a common neighbor other than $v$. Suppose $uw$ is a boundary edge of $G$, then by Observation \ref{obs:NoConsecutive4} either $u$ or $w$ is a $3$-degree vertex. Without loss of generality, let $w$ be the $3$-degree vertex. Then the cyclically next boundary neighbor of $w$ other than $u$, say $x$ is also a neighbor of $v$. At least one vertex in $\{u,w,x\}$ will be included in $I$ by the end of Round~2. Hence $v$ would not have been added to $I$. 
    
    Since these are the only types of vertices in $I$ (Observation~\ref{obs:Ivertices}), the above cases are exhaustive. 
\end{proof}

\begin{remark}
    If $G$ is a triangulation of order at least $4$, Observation~\ref{obs:CommonSecondneigbor} can be directly established  for any $4^-$-vertex $v$, since every edge is part of two triangles. In that case, we can pick $I$ to be any maximal independent set of $4^-$-vertices in $G$ and skip all previous observations made in this section. This would have sufficed if our aim was limited to affirming Corollary~\ref{corr:GodHenTriangulation}.
\end{remark}

    Observation~\ref{obs:CommonSecondneigbor} leads us to a simple idea which helped us unlock this problem. 
    
\begin{observation}[Key Observation]                    \label{obs:Key}
    If there exists a two-coloring $f$ of $V(G)$ such that every vertex in $V(G) \setminus N(I)$ is satisfied, then $f$ can be modified to a two-coloring which satisfies every vertex in $G$.
\end{observation}
\begin{proof}
     Suppose there is a vertex $v \in I$ such that a vertex in $N(v)$ is unsatisfied. Let $G_v = G \setminus \set{v}$ and $f_v$ be $f$ restricted to $V(G_v)$. Let $U \subset N(v)$ be the set of unsatisfied vertices in $N(v)$ under $f_v$. Note that this may contain vertices which were satisfied under $f$. If $\card{U} \geq 2$, by Observation~\ref{obs:CommonSecondneigbor}, each pair of vertices in $U$ have a common neighbor in $G_v$ and hence miss the same color under $f_v$ (the color different from that of the common neighbor). Since this is true for every pair in $U$, all the vertices in $U$ miss the same color, say $c$, under $f_v$. If $\card{U} = 1$, we choose $c$ to be the color missing for the single vertex $u \in U$ under $f_v$. 
     
     Recoloring $v$ to $c$ in $f$ gives a new two-coloring $f'$ of $V(G)$ which satisfies all the vertices in $U$ and all the vertices originally satisfied by $f$. We can repeat this procedure till no vertex in $I$ has an unsatisfied neighbor to get a $2$-coupon coloring of $G$.
\end{proof}
     
In view of Observation~\ref{obs:Key}, we can focus on finding a two-coloring of $G$ such that every vertex in $V' = V(G) \setminus N(I)$ is satisfied. The set $V'$ consists of two types of vertices - those in $I$ and those not dominated by $I$. Those in $I$ are all $4^-$-vertices and those not dominated by $I$ are $5^+$-vertices in $G$ and remain so even in $G \setminus I$. These two types of vertices pose different challenges which can be jointly addressed by a new type of coloring for a triangulation $G'$ which contains $G \setminus I$ as a spanning subgraph. First we construct the graph $G'$ and then define the new type of coloring.

\begin{construction}[$(G',P')$]                          \label{const:Gprime}
    Let $v_1, \ldots, v_k$ be an arbitrary order of vertices in $I$ and let $G_0 = G$ and $P_0 =  \emptyset$. For each $i \in [k]$, the graph $G_i$ is obtained from $G_{i-1}$ by deleting $v_i$ and adding any missing edge between the neighbors of $v_i$ in $G_{i-1}$ as long as they do not violate planarity. Add all the edges of $G_i$ between the neighbors of $v_i$ in $G_{i-1}$ to $P_{i-1}$ to obtain $P_i$. Let $(G', P') = (G_k, P_k)$. We call the edges in $P'$ and their endpoints \emph{protected} and the remaining vertices \emph{unprotected}.
\end{construction}

\begin{observation}                                     \label{obs:Unprotected}
    Every unprotected vertex in $G'$ is a $5^+$-vertex in $G'$.
\end{observation}
\begin{proof}
    Since $I$ is a maximal independent set of $4^-$-vertices, every $4^-$-vertex $v$ is in $I \cup N(I)$. If $v \in I$ it gets deleted and if $v \in N(I)$, it gets protected. Hence every unprotected vertex $u$ in $G'$ is a $5^+$-vertex in $G$. Since $u \notin N(I)$, it remains a $5^+$-vertex in $G \setminus I$ and in $G'$ which is a supergraph of $G \setminus I$.
\end{proof}

\begin{observation}                                     \label{obs:GprimeTriangle}
    If $v \in I$ is a $3^+$-vertex in $G$, $\indsg{N_G(v)}$ contains a triangle $T_v$ in $G'$ and all the three edges of $T_v$ are in $P'$. If $v \in I$ is a $2$-vertex in $G$, the edge in $\indsg{N_G(v)}$ is in $P'$. 
\end{observation}
\begin{proof}
    For every vertex $v_i \in I$, $N(v_i) \subset V(G')$ since $V(G') = V(G) \setminus I$ and $I$ is an independent set. If $v_i$ is an internal $3$-vertex of $G$, then $N_G(v_i)$ induces a triangle in $G$. If $v_i$ is a boundary $3$-vertex of $G$ and $N_{G_{i-1}}(v_i)$ induced only a $2$-length path, it can be completed to a triangle in $G_i$ by connecting the two boundary neighbors of $v_i$ through the outer face of $G_{i-1}$. Finally, if $v_i$ is an internal $4$-vertex, $N_{G_{i-1}}(v_i)$ contains at least one pair of non-adjacent vertices in $G_{i-1}$ (since $G_{i-1}$ is $K_5$-free) and they can be connected in $G_i$ by an edge through the $4$-face created by deleting $v_i$ from $G_{i-1}$. Notice that no edge between two vertices in $N_G(v_i)$, whether originally present in $G$ or added in one of the steps, gets deleted later since $I$ is an independent set. The required memberships in $P'$ follow from the construction.  
\end{proof}

\begin{definition}[Fair four-coloring]
    Given a graph $G$ and $P \subset E(G)$, a four-coloring of $V(G)$ is called a \emph{fair four-coloring} of $(G, P)$ if the endpoints of every edge in $P$ gets different colors and every vertex $v$ not spanned by $P$ sees at least three colors in $N(v)$.
\end{definition}

\begin{lemma}                                           \label{lem:fair-coloring}
     If $\Gamma$ is a planar graph and $P \subseteq E(\Gamma)$ spans all $4^-$-vertices in $\Gamma$, then $(\Gamma, P)$ has a fair four-coloring.
\end{lemma}
\begin{proof}
     Borrowing the terminology from Construction~\ref{const:Gprime}, we call the edges in $P$ and their endpoints \emph{protected} and the remaining vertices \emph{unprotected}. Consider a cut $(A,B)$ in $\Gamma$ with maximum number of edges crossing the parts subject to the constraint that no protected edge crosses the cut. By the maximality of the cut, every unprotected vertex $v$ will have at least half of its neighbors in the opposite part. Otherwise, we can shift $v$ to the other part to get a larger cut without violating the constraint. Since all unprotected vertices in $G$ are $5^+$-vertices, they will have at least three neighbors on the other side. 
     
     We color $V(\Gamma)$ by coloring $A$ and $B$ independently, starting with $A$. Remove all the edges between the vertices in $B$. If there are more than three neighbors for any unprotected vertex $v \in B$, then arbitrarily delete some edges incident on $v$ until exactly three edges remain. Call the resulting subgraph $\Gamma'$. Pick a planar drawing of $\Gamma'$ and for each unprotected vertex $v \in B$, add any missing edge in the $A$-part between every pair of vertices in $N_{\Gamma'}(v)$. This can be done without violating planarity since each unprotected vertex in $B$ has only three neighbors in $\Gamma'$. Let us call this graph $H$ and the planar subgraph of $H$ induced on $A$ as $H_A$. By the four color theorem \cite{AppHak,AppHakKoc}, there exists a proper four-coloring $f_A$ of $H_A$. Repeat the same procedure to find a coloring $f_B$ of the vertices in $B$ and combine them to get a coloring $f$ of $\Gamma$.  
     
     For every unprotected vertex $v\in B$ (resp., $v \in A$), there will be a triangle induced in $H_A$ (resp., $H_B$) by the three neighbors of $v$ in $H$. Hence $N(v)$ will see 3 different colors in $f$. Since none of the edges in $P$ crossed the cut and since all of them where present in either $H_A$ or $H_B$, the endpoints of every protected edge gets different colors. Hence $f$ is a fair four-coloring of $(\Gamma, P)$.
\end{proof}

\begin{remark}
    The idea of using a max-cut (without any constraint) was used by B\'{e}rczi and G\'{a}bor \cite{bergab} to prove that triangulations with at most two vertices of degree at most four have a $2$-coupon coloring.
\end{remark}

By Lemma~\ref{lem:fair-coloring}, the pair $(G', P)$ obtained from Construction~\ref{const:Gprime} has a fair four-coloring $f$. We construct a two-coloring of $G$ from $f$ as follows.

\begin{construction}[$f_2$]                             \label{const:f2}
Consider $f$ as a partial four-coloring of $G$. Note that for each $2$-vertex $v \in I$, $N(v)$ contains a protected edge (Observation~\ref{obs:GprimeTriangle}) and hence $N(v)$ sees two colors under $f$. Group the four colors into two pairs so that for each $2$-vertex $v$ in $I$, the two colors on the neighbors of $v$ go to different pairs. This is indeed possible since we have at most two $2$-vertices in $G$. Now merge the two colors in a pair into a single color. Extend this partial two-coloring to a full two-coloring $f_2$ of $G$ by giving the vertices in $I$ any of the two colors arbitrarily. 
\end{construction}

\begin{observation}                                     \label{obs:f2}
    In the two-coloring $f_2$ obtained from Construction~\ref{const:f2}, every vertex in $V(G) \setminus N(I)$ is satisfied.
\end{observation}
\begin{proof}
    Recall that $V(G) \setminus N(I)$ consists of two types of vertices - the unprotected vertices in $G'$ and the vertices in $I$. If $v$ is an unprotected vertex,  since $f$ is a fair four-coloring, $N(v)$ contains vertices of at least three colors under $f$ and hence two colors under $f_2$. If $v$ is a $2$-vertex in $I$, then it is satisfied due to the careful merging of colors in Construction~\ref{const:f2}. If $v$ is a $3^+$-vertex in $I$, then $N(v)$ in $G'$ contains a triangle $T_v$, all of whose edges are protected (Observation~\ref{obs:GprimeTriangle}). Hence $N(v)$ sees at least three colors under $f$ and hence two colors under $f_2$. 
\end{proof}

Observation~\ref{obs:f2} says that the $f_2$ satisfies the premise of Observation~\ref{obs:Key} and hence we can conclude that $f_2$ can be modified to a two-coloring which satisfies every vertex in $G$. Hence $G$ is good. This completes the proof of Theorem~\ref{thm:NearCouponColouring} and equivalently Theorem~\ref{thm:General}.

\begin{remark}
    It should be clear by now that key role played by the proper four-coloring of $H_A$ and $H_B$ is in ensuring that no triangle is monochromatic after the merger into a two-coloring. The existence of such two-colorings can be proved without resorting to the four-color theorem (c.f. Kaiser and {\v{S}}krekovski, 2004 \cite{kaiser2004planar}, Thomassen, 2008 \cite{thomassen20082}). Perhaps the easiest way (due to Barnette) is to use a stronger version of Petersen's theorem which asserts that every edge of a bridgeless cubic multigraph is contained in a $1$-factor (Sch{\"o}nberger, 1934 \cite{schonberger1934beweis}). Hence we could have bypassed the use of four-color theorem if we did not have to handle the two $2$-vertices. In particular, we can prove Corollary~\ref{corr:GodHenTriangulation} without using the four-color theorem. 
\end{remark}

\section{Concluding Remarks}

    Our proof of Theorem~\ref{thm:NearCouponColouring} lends itself to a polynomial-time algorithm. Notice that even though finding a max-cut is NP-hard for general graphs, it is polynomial-time solvable for planar graphs. Moreover, we do not even need a max-cut for our purpose. As soon as we get a cut which cannot be improved by shifting one unprotected vertex to the opposite part, we are done. This can be done greedily and will terminate in at most as many rounds as the number of unprotected edges. Four coloring of a planar graph can be done in quadratic time \cite{robertson1996efficiently}. 

    While we were able to affirm two of the conjectures in \cite{god}, we could not solve a tantalizing strengthening which states that the vertex set of a triangulation $G$ with at least four vertices can be partitioned into two total dominating sets - both of which induce a bipartite subgraph of $G$. Equivalently, there exists a proper four-coloring with color classes $\set{V_1, V_2, V_3, V_4}$ such that     both $V_1 \cup V_2$ and $V_3 \cup V_4$ are total dominating sets (Conjecture~32, \cite{god}). Our method seems to be limited in power when we need a proper coloring.
    
    
    Since we have affirmed two conjectures in this paper, we wish to restore the balance by posing two of our own. The first one stems out of the key technique we used in our proof and the second one comes out of our attempts to refute the original conjecture which we ended up proving.

\begin{conjecture}                                     \label{conj:4color}
    Every near-triangulation has a four-coloring of its vertices such that every vertex $v$ sees at least $\min\{d(v), 3\}$ different colors in $N(v)$.
\end{conjecture}

Conjecture~\ref{conj:4color} will immediately give Theorem~\ref{thm:NearCouponColouring} via the color merger argument we used in Construction~\ref{const:f2}. We anticipate that, if proven, this may find many more applications than Theorem~\ref{thm:NearCouponColouring}. Some of the earlier attempts to settle the Goddard-Henning conjecture on certain classes of triangulations can be modified to affirm Conjecture~\ref{conj:4color} for those classes. For example, once can see that triangulations with acyclic chromatic number at most four (this includes triangulations with all vertex degrees odd) satisfy Conjecture~\ref{conj:4color}. 

\begin{conjecture}                                     \label{conj:allplanar}
    If $G$ is a planar graph with minimum degree at least three, then $d_t(G) \geq 2$.
\end{conjecture}

A look in to the coloring part in the proof of Lemma~\ref{lem:fair-coloring} will show that if $G$ is a planar graph which has a cut such that every vertex $v$ has at least three neighbors in the opposite part, then $(G, \emptyset)$ has a fair four-coloring and hence $G$ has a $2$-coupon coloring. This suffices to confirm Conjecture \ref{conj:allplanar} for all planar graphs with minimum degree at least five and all bipartite planar graphs.

\bibliographystyle{ams}
\bibliography{bibtex}  
\end{document}